\documentclass{amsart}

\usepackage{amsmath}
  \usepackage{paralist}
  \usepackage{amssymb}
 \usepackage{amsthm}
 \usepackage[colorlinks=true]{hyperref}
\hypersetup{urlcolor=blue, citecolor=red}

\newtheorem{theorem}{Theorem}[section]

\newtheorem{lemma}[theorem]{Lemma}

\theoremstyle{definition}

\newtheorem*{remark}{Remark}

\numberwithin{equation}{section}

\title[Sign-changing solutions for Kirchhoff type problems]
      {Infinitely many sign-changing solutions for Kirchhoff type problems in $\mathbb{R}^3$}

\author[Jijiang Sun, Lin Li,
M. Cencelj, B. Gabrov\v{s}ek]{}


 \email{sunjijiang2005@163.com} \email{lilin420@gmail.com}  \email{matija.cencelj@fmf.uni-lj.si} \email{bostjan.gabrovsek@fmf.uni-lj.si}

\thanks{J. Sun
was supported by NSFC (No.11501280, No.11861046) and the Natural Science Foundation of Jiangxi Province (No.20181BAB201004). L. Li was supported by the National Natural Science Foundation of China (No. 11601046), Chongqing Science and Technology Commission (No. cstc2016jcyjA0310) and Program for University Innovation Team of Chongqing (No. CXTDX201601026).
M. Cencelj and B. Gabrov\v{s}ek
were supported by the Slovenian Research Agency grants
J1-8131, J1-7025, N1-0064 and N1-0083.}
\thanks{$^*$ Corresponding author}

\date{\today}

\begin{document}
\maketitle
\medskip
\centerline{\scshape Jijiang Sun}
\medskip
{\footnotesize
 \centerline{Department of Mathematics}
   \centerline{Nanchang University, Nanchang 330031, PR China}}

\medskip
\centerline{\scshape Lin Li}
\medskip
{\footnotesize
 \centerline{School of Mathematics and Statistics}
   \centerline{Chongqing Technology and Business University, Chongqing 400067, PR China}}

\medskip
\centerline{\scshape
	Matija Cencelj}
\medskip
{\footnotesize
	\centerline{Faculty of Education and Faculty of Mathematics and Physics}
	\centerline{University of Ljubljana, 1000 Ljubljana, Slovenia}}
	
\medskip
\centerline{\scshape
Bo\v{s}tjan Gabrov\v{s}ek}
\medskip
{\footnotesize
	\centerline{Faculty of Mechanical Engineering and Faculty of Mathematics and Physics}
	\centerline{University of Ljubljana, 1000 Ljubljana, Slovenia}}
\bigskip
 %\centerline{(Communicated by the associate editor name)}

\begin{abstract}
In this paper, we consider the following nonlinear Kirchhoff type problem:
\[
\left\{\begin{array}{lcl}-\left(a+b\displaystyle\int_{\mathbb{R}^3}|\nabla u|^2\right)\Delta u+V(x)u=f(u), & \textrm{in}\,\,\mathbb{R}^3,\\
u\in H^1(\mathbb{R}^3),
\end{array}\right.
\]
where $a,b>0$ are constants, the nonlinearity $f$ is superlinear at infinity with subcritical growth and $V$ is continuous and coercive.
For the case when
 $f$ is odd in $u$
 we obtain infinitely many sign-changing solutions for the above problem
 by using a combination of invariant sets method and
 the Ljusternik-Schnirelman type minimax method.
To the best of our knowledge, there are only few existence results for this problem.
It is worth mentioning that the nonlinear term may not
be 4-superlinear at infinity, in particular, it includes the power-type nonlinearity
$|u|^{p-2}u$ with $p\in(2,4]$.
\end{abstract}

\section{Introduction and Main Results}

In this paper we are interested in establishing the multiplicity of sign-changing solutions to the following nonlinear Kirchhoff type problem
\begin{equation}\label{eq1}
\left\{\begin{array}{lcl}-\left(a+b\displaystyle\int_{\mathbb{R}^3}|\nabla u|^2\right)\Delta u+V(x)u=f(u), & \textrm{in}\,\,\mathbb{R}^3,\\
u\in H^1(\mathbb{R}^3),
\end{array}\right.
\end{equation}
where $a,b>0$ are constants, $V\in C(\mathbb{R}^3,\mathbb{R})$, and $f\in C(\mathbb{R},\mathbb{R})$.

Problems like \eqref{eq1} have been widely investigated because they have a strong physical meaning. Indeed, \eqref{eq1} is related to the stationary analogue of the equation
\begin{equation}\label{eq2}
u_{tt}-\left(a + b \int_{\Omega}|\nabla u|^2dx \right) \Delta u=f(x,u)
\end{equation}
proposed by Kirchhoff in \cite{Kirchhoff1883} as an extension of the classical D'Alembert's wave equation for free vibrations of elastic strings. In \cite{Lions1977}, Lions proposed an abstract framework for
 the problem and after that, problem \eqref{eq2} began to receive
 a lot of
 attention.

In \eqref{eq1}, if we set $V(x)=0$ and replace $\mathbb{R}^3$ and $f(u)$ by a bounded domain $\Omega\subset\mathbb{R}^N$ and $f(x,u)$, respectively, then we get the following Kirchhoff type equation
\begin{equation}\label{eq3}
\left\{\begin{array}{lcl}-\left(a+b\displaystyle\int_{\Omega}|\nabla u|^2\right)\Delta u=f(x,u), & x\in\Omega,\\
u=0,& x\in\partial\Omega.
\end{array}\right.
\end{equation}
The above problem is a nonlocal one as the appearance of the term $\int_{\Omega}|\nabla u|^2dx$ implies that \eqref{eq3} is not a pointwise identity. This phenomenon causes some mathematical difficulties, which make the study of \eqref{eq3} particularly interesting. In recent years, by using variational methods, the solvability of equation \eqref{eq3} with subcritical or critical growth nonlinearity has been paid much attention by various authors, see, e.g. \cite{Chen2011,Ma2003,Naimen2014,Perera2006,Sun2011,Sun2012} and the references therein. For the results concerning the existence of sign-changing solutions for \eqref{eq3}, we refer the reader to papers \cite{Mao2009,Shuai2015,Zhang2006} which depend heavily on the nonlinearity term with $4$-superlinear growth at infinity in the sense that $$\lim_{|t|\to\infty}\frac{F(x,t)}{t^4}=+\infty,\quad\textrm{uniformly\ in\ }x\in\Omega,$$
where $F(x,t)=\int_0^tf(x,s)ds$ and \cite{Lu2015,Yao2016} with the nonlinearity $f(x,u)$ may not be $4$-superlinear at infinity.

If we replace $f(u)$ by $f(x,u)$ in \eqref{eq1},
several authors have
considered the following problem
\begin{equation}\label{eq4}
-\left(a+b\displaystyle\int_{\mathbb{R}^N}|\nabla u|^2\right)\Delta u+V(x)u=f(x,u), \quad x\in\mathbb{R}^N.
\end{equation}
In recent years, there have been enormous results on existence, nonexistence and
multiplicity of nontrivial solutions for such problem depending on the assumptions of the potential $V$ and $f$. See, for example, \cite{Li2014,Liy2012,Wu2011} and the references therein. Recently, replacing $a$ and $b$ by $\varepsilon^2a$ and $\varepsilon b$ in \eqref{eq4}, respectively, many researches have studied
a certain concentration phenomena for the following Kirchhoff type equation
\begin{equation*}
\left\{\begin{array}{lcl}\left(-\varepsilon^2a+\varepsilon b\displaystyle\int_{\mathbb{R}^3}|\nabla u|^2\right)\Delta u+V(x)u=f(x,u),& \textrm{in}\,\,\mathbb{R}^3,\\
u>0,u\in H^1(\mathbb{R}^3),
\end{array}\right.
\end{equation*}
see e.g. \cite{Figueiredo2014,He2012,He2015,Lius2015,Wang2012}. We mention that there are only
few works concerning
the existence of sign-changing solutions for \eqref{eq4}. We are only aware of the works \cite{Deng2015,Huang2014,Ye2015}. In \cite{Deng2015}, Deng \textit{et al.} studied the existence of radial sign-changing solutions with prescribed numbers of nodal domains for \eqref{eq4} in $H_r^1(\mathbb{R}^3)$, the subspace of radial functions of $H^1(\mathbb{R}^3)$ by using a Nehari manifold and gluing solution pieces together, when $V(x)=V(|x|)$ satisfies
\begin{itemize}
	\item [$(V'_1)$] $V\in C([0,+\infty),\mathbb{R})$ is bounded below by a positive constant $V_0$;
\end{itemize}
and $f(x,u)=f(|x|,u)$ satisfies the following hypotheses:
\begin{itemize}
	\item [$(f'_{1})$] $f(r,u)\in C^1([0,+\infty)\times\mathbb{R},\mathbb{R})$ is odd in $u$ for every $r\geq0$;
	\item [$(f'_{2})$] $f(r,u)=o(|u|)$ as $u\to0$ uniformly in $r\geq0$;
	\item [$(f'_{3})$] for some constant $p\in(4,6)$, $\lim_{u\to+\infty}\frac{f(r,u)}{u^{p-1}}=0$ uniformly in $r\geq0$;
	\item [$(f'_{4})$] $\lim_{u\to+\infty}\frac{F(r,u)}{u^4}=+\infty$, where $F(r,u)=\int_{0}^{u}f(r,t)dt$;
	\item [$(f'_{5})$] $\frac{f(r,u)}{|u|^3}$ is an increasing function of $u\in\mathbb{R}\setminus\{0\}$ for every $r\geq0$.
\end{itemize}
In \cite{Huang2014}, Huang and Liu studied the existence of least energy sign-changing solutions with exactly two nodal domains for a variant of \eqref{eq4}:
\begin{equation*}
-\left(1+\lambda\displaystyle\int_{\mathbb{R}^N}(|\nabla u|^2+V(x)u^2)\right)[\Delta u+V(x)u]=|u|^{p-2}u, \quad x\in\mathbb{R}^N.
\end{equation*}
where $\lambda>0$, $p\in(4,6)$ and $V$ is assumed to guarantee the compactness. Ye \cite{Ye2015} proved
 the existence of least energy sign-changing solutions for equation \eqref{eq4} with $f(x,u)=f(u)$ (i.e., \eqref{eq1}) by using constrained minimization of the sign-changing Nehari manifold and Brouwer degree theory under the conditions that $V$ satisfies
\begin{itemize}
	\item [$(V_1)$] $V\in C(\mathbb{R}^3,\mathbb{R})$ satisfies $\inf_{\mathbb{R}^3}V(x)\geq V_0>0$ for some positive constant $V_0$ and is coercive, i.e., $\lim\limits_{|x|\to\infty}V(x)=\infty$,
\end{itemize}
and the nonlinearity $f\in C^1(\mathbb{R},\mathbb{R})$ satisfies the following assumptions:
\begin{itemize}
	\item [$(\widetilde{f}_{1})$] $\lim_{s\to0}\frac{f(s)}{|s|^3}=0$;
	\item [$(\widetilde{f}_{2})$] there exists $3<q<2^*-1$ such that $\lim_{|s|\to+\infty}\frac{f(s)}{|s|^q}=0$, where $2^*=+\infty$ if $N=2$ and $2^*=6$ if $N=3$;
	\item [$(\widetilde{f}_{3})$] $\lim_{|s|\to+\infty}\frac{F(s)}{s^4}=+\infty$, where $F(s)=\int_{0}^{s}f(t)dt$;
	\item [$(\widetilde{f}_{4})$] the function $\frac{f(s)}{|s|^3}$ is nondecreasing on $\mathbb{R}\setminus\{0\}$.
\end{itemize}

To the best of
 our knowledge, there is no result in the literature on the existence of multiple sign-changing solutions for problems
  \eqref{eq1} and \eqref{eq4} without any symmetry. Motivated by the above works, in the present paper we
   study the existence of infinitely many sign-changing solutions for problem \eqref{eq1} with coercive potential $V$, that is,
    $(V_1)$ holds and more general assumptions on $f$. More precisely, we assume
    that
     $f$ satisfies the following assumptions:
\begin{itemize}
	\item [$(f_{1})$] $f\in C(\mathbb{R},\mathbb{R})$ and $|f(u)|\leq C(1+|u|^{p-1})$ for some $C>0$ and $p\in(2,6)$;
	\item [$(f_{2})$] $f(u)=o(u)$ as $u\to0$;
	\item [$(f_{3})$] there exists $\mu>2$ such that $\frac{1}{\mu}f(u)u\geq F(u)>0$ for all $u\in\mathbb{R}\setminus\{0\}$, where $F(u)=\int_{0}^{u}f(s)ds$;
	\item [$(f_{4})$] $f$ is odd, i.e., $f(-u)=-f(u)$.
\end{itemize}

Now we state our first main result.

\begin{theorem}\label{thm1.1}
	Suppose
	that
	$(V_1)$ and $(f_{1})$--$(f_{4})$ hold and $\mu>4$.
	Then problem \eqref{eq1} admits infinitely many sign-changing solutions.
\end{theorem}

\begin{remark}
	The assumptions $(V_1)$ plays a role only in guaranteeing the compactness of the (PS) sequence for the energy functional $I$ associated with \eqref{eq1}. We point out that Theorem \eqref{thm1.1} also holds when working in $H_r^1(\mathbb{R}^3)$ if $V$ is a positive constant.
\end{remark}

Recall that $(f_{3})$ is the so-called Ambrosetti-Rabinowitz condition ((AR) for short). It is easy to see that $\mu>4$ guarantees the Palais-Smale ((PS) for short) sequence for $I$ at any $c\in\mathbb{R}$ is bounded. But if $\mu<4$, $f$ may not be $4$-superlinear at infinity, due to the effect of the nonlocal term, it is difficult to get a bounded (PS) sequence for $I$. Motivated by \cite{Li2014}, to overcome this difficulty, in the case $\mu<4$, we suppose that $V(x)$
satisfies
 the following additional condition
\begin{itemize}
	\item [$(V_2)$] $V$ is weakly differentiable, $(DV(x),x)\in L^{r}(\mathbb{R}^3)$ for some $r\in[\frac{3}{2},\infty]$ and $$\frac{\mu-2}{2}V(x)-(DV(x),x)\geq0\quad\textrm{for\, a.e.\ }\,x\in\mathbb{R}^3,$$
	where $\mu$ is given by $(f_{3})$.
\end{itemize}
It is worth mentioning that this assumption is different from that of \cite{Li2014}. Li and Ye \cite{Li2014} assumed $$V(x)-(DV(x),x)\geq0\quad\textrm{for\, a.e.\ }\,x\in\mathbb{R}^3,$$ and then obtained a positive ground state solution to \eqref{eq1} with $f(u)=|u|^{p-2}u$ $(p\in(3,6))$ by using the constrained minimization on a suitable Pohozaev-Nehari manifold. We remark that the case $2<p\leq 3$ is not included in their result.

Then we have the following result.
\begin{theorem}\label{thm1.2}
	Suppose
	that  $(V_1)$--$(V_2)$ and $(f_{1})$--$(f_{4})$ hold.
	Then problem \eqref{eq1} admits infinitely many sign-changing solutions.
\end{theorem}

\begin{remark}
	(i) To the best of our knowledge, there is no existence result for sign-changing solutions to \eqref{eq1} in the literature even in the special case $f(u)=|u|^{p-2}u$ with $2<p\leq4$.
	
	(ii) There exists function $V(x)$ satisfying the assumptions $(V_1)$--$(V_2)$. For example, let $$V(x)=\ln(1+|x|)+\frac{2}{\mu-2}.$$
	Clearly, $(V_1)$ holds. Moreover, for $x\in\mathbb{R}^3\setminus\{0\}$, $(DV(x),x)=\frac{|x|}{1+|x|}$. Therefore, $(DV(x),x)\in L^{\infty}(\mathbb{R}^3)$ and for a.e. $x\in\mathbb{R}^3$, $$\frac{\mu-2}{2}V(x)-(DV(x),x)=\frac{\mu-2}{2}\ln(1+|x|)+1-\frac{|x|}{1+|x|}\geq\frac{\mu-2}{2}\ln(1+|x|)\geq0,$$
	and so condition $(V_2)$ holds. Another example is $V(x)=\ln(1+|x|^2)-\frac{|x|^2}{1+|x|^2}+\frac{\mu+2}{\mu-2}$. One can also check that $V(x)$ satisfies $(V_1)$--$(V_2)$.
\end{remark}

Motivated by \cite{Sun2017,Sun2014}, we will prove Theorems \ref{thm1.1} and \ref{thm1.2} by applying the usual Ljusternik-Schnirelman type minimax method in conjunction with invariant set method. More precisely, we will construct certain invariant sets of the gradient flow corresponding to the energy functional $I$ such that all positive and negative solutions are contained in these invariant sets and then minimax arguments can be applied to construct sign-changing solutions outside these invariant sets. The method of invariant sets of descending flow has been used widely in dealing with sign-changing solutions of elliptic problem, see \cite{Bartsch2004,Bartsch2005,Liu2001} and the references therein. But for the nonlocal problem, there are serious technical
 difficulties to overcome. Here we would like to point out the difficulties we will encounter and our main ideas.

Due to the effect of the nonlocal term, the arguments of constructing invariant sets of descending flow in \cite{Bartsch2004,Liu2005,Sun2014} cannot be directly applied to problem \eqref{eq1}. To overcome this difficulty, we will adopt some ideas from \cite{Liu2015} which studied the existence of infinitely many sign-changing solutions for a Schr\"{o}dinger-Poisson system by using an abstract critical point developed by Liu et al. \cite{Liuj2015}. First,
 we will construct an auxiliary operator $A$ (see Lemma \ref{lem3.00} below) from which we can construct closed convex sets containing all the positive and negative solutions in their interior. However, $A$ itself cannot be used to defined the desired invariant sets of the flow, because the operator $A$ is merely continuous under our assumptions. Inspired by \cite{Bartsch2005}, from $A$, we can get a locally Lipschitz continuous operator $B$ (see Lemma \ref{lem3.3} below) which inherits the main properties of $A$. Then, we can use $B$ to define the flow (see Lemma \ref{lem3.4} below). Finally, by using a suitable deformation lemma in the presence of invariant sets (see Lemma \ref{lem3.5} below) and minimax procedures, we prove that problem \eqref{eq1} has infinitely many sign-changing solutions.

As mentioned above, if $\mu<4$, the nonlinearity term $f$ may not be $4$-superlinear at infinity. It prevents us from obtaining a bounded (PS) sequence, let alone (PS) condition holds for $I$. Therefore, the above arguments cannot be applied directly to prove Theorem \ref{thm1.2}. To overcome the obstacle, inspired by \cite{Liu2015}, we consider the perturbed functionals $I_\lambda: E\to\mathbb{R}$ (see \eqref{eq4.000} below) defined by $$I_\lambda(u)=I(u)-\frac{\lambda}{r}\int_{\mathbb{R}^3}|u|^r, \quad \lambda\in(0,1],$$
where $r\in(\max\{4,p\},6)$, here $p$ is from $(f_{1})$. It will be shown that $I_\lambda$ admits infinitely many sign-changing critical points $\{u_k^\lambda\}_{k\geq2}$ by using the above framework. Then, by using a Pohozaev identity and $(V_2)$, we can prove that $u_k^\lambda\to u_k$ strongly in $E$ as $\lambda\to0^+$ and then the existence of infinitely many sign-changing solutions for \eqref{eq1} are obtained.

The remainder of this paper is organized as follows. In
Section \ref{sec2} we derive a variational setting for
problem \eqref{eq1} and give some preliminary lemmas. We will prove Theorem \ref{thm1.1} in
Section \ref{sec3}. Section \ref{sec4} is devoted to the proof of Theorem \ref{thm1.2}.

\section{Variational setting and preliminary lemmas}\label{sec2}

Throughout this paper, we use the standard notations. We denote by $C,c_i,C_i,i=1,2,\cdots$ for various positive
constants whose exact value may change from lines to lines but are
not essential to the analysis of problem. $\|\cdot\|_{q}$ denotes the usual norm of
$L^{q}(\mathbb{R}^3)$ for $q\in[2,\infty]$. For simplicity, we write $\int_{\mathbb{R}^3}h$ to mean the Lebesgue integral of $h(x)$ over $\mathbb{R}^3$. For a functional $J: E\to \mathbb{R}$, we set
$J^b:=\{u\in E:J(u)\leq b\}$. We use ``$\rightarrow$" and ``$\rightharpoonup$" to denote the strong and weak convergence in the related function space respectively.
We will write $o(1)$ to denote quantity that tends to $0$ as $n\to\infty$.

Our argument is variational. In the paper, we work in the following Hilbert space $$E:=\left\{u\in H^1(\mathbb{R}^3):\int_{\mathbb{R}^3} V(x)u^2<\infty\right\}$$
with the norm $$\|u\|_E=\left(\int_{\mathbb{R}^3}(a|\nabla u|^2+V(x)|u|^2)\right)^{\frac{1}{2}}.$$ We denote its inner product by $(\cdot,\cdot)_E$. It is well known that $E\hookrightarrow L^s(\mathbb{R}^3)$ is continuous for $s\in[2,6]$. Moreover, as in \cite{Bartsch1995}, we have the following result which plays an important role in our proof.

\begin{lemma}\label{lem2}
	Under $(V_1)$, the embedding $E\hookrightarrow L^s(\mathbb{R}^3)$ is compact for any $s\in[2,6)$.
\end{lemma}

\begin{remark}
	Indeed, as in \cite{Bartsch2001}, $(V_1)$ can be replaced by the more general condition.
	\begin{itemize}
		\item [$(V'_1)$]  $V\in C(\mathbb{R}^3,\mathbb{R})$ satisfies $\inf_{\mathbb{R}^3}V(x)\geq V_0>0$ and there exists $r_0>0$ such that for any $M>0$,
		$$\lim_{|y|\to\infty}m\left(\{x\in\mathbb{R}^3:V(x)\leq M\}\cap\{x\in\mathbb{R}^3:|x-y|\leq r_0\}\right)=0.$$
	\end{itemize}
\end{remark}

Since $E\hookrightarrow L^2(\mathbb{R}^3)$ and $L^2(\mathbb{R}^3)$ is a separable Hilbert space, $E$ has a countable orthogonal basis $\{e_i\}$. In the following, for any $m\in\mathbb{N}$, we denote $E_m:=\textrm{span}\{e_1,e_2,\cdots,e_m\}$.

Under our assumptions, it is standard to show that the weak solutions to \eqref{eq1} correspond to the critical points of the energy functional $I\in C^1(E,\mathbb{R})$ defined by
\begin{equation}\label{eq2.0}
I(u)=\frac{1}{2}\int_{\mathbb{R}^3}\left(a|\nabla u|^2+V(x)|u|^2\right)+\frac{b}{4}\left(\int_{\mathbb{R}^3}|\nabla u|^2\right)^2-\int_{\mathbb{R}^3}F(u),
\end{equation}
where $F(u)=\int_0^uf(s)ds$. Moreover, for any $u,v\in E$, we have
$$\langle I'(u),v\rangle=\int_{\mathbb{R}^3}\left(a\nabla u\cdot\nabla v+V(x)uv\right)+b\int_{\mathbb{R}^3}|\nabla u|^2\int_{\mathbb{R}^3}\nabla u\cdot\nabla v-\int_{\mathbb{R}^3}f(u)v.$$

\section{proof of Theorem \ref{thm1.1}}\label{sec3}

In this section, we devote to prove the existence of infinitely many sign-changing solutions to problem \eqref{eq1} with $\mu>4$ by using a combination of invariant sets method and Ljusternik-Schnirelman type minimax results.

\subsection{Some technical lemmas}

In order to construct the minimax values for the functional $I$ defined in \eqref{eq2.0}, the following three technical lemmas are needed.

\begin{lemma}
	Under the assumptions $(f_1)$-$(f_3)$, the functional $I$ satisfies the $\rm(PS)$ condition.
\end{lemma}

\begin{proof}
	By $(f_3)$, it is easy to check that any (PS) sequence for $I$ at level $c\in\mathbb{R}$ is bounded. Thus, by Lemma \ref{lem2}, one can follow the same way as in the proof of Lemma 4 in \cite{Wu2011} to complete the present proof.
\end{proof}

\begin{lemma}\label{lem3}
	Suppose $(f_1)$-$(f_3)$ hold and $m\geq1$. Then there exists $R_m=R(E_m)>0$, such that
	$$\sup_{\mathcal{B}^{c}_{R_m}\cap E_m}I<0,$$
	where $\mathcal{B}^{c}_{R}:=E\setminus \mathcal{B}_{R}$ and $\mathcal{B}_{R}:=\{u\in E:\|u\|_E\leq R\}$.
\end{lemma}

\begin{proof}
	By $(f_1)$-$(f_3)$, there exists $c_1,c_2>0$ such that $F(t)\geq c_1|t|^\mu-c_2|t|^2$ for all $t\in\mathbb{R}$. Thus,
	\begin{equation*}
	\begin{split}
	I(u)&\leq\frac{1}{2}\|u\|^2_{E}+\frac{b}{4a^2}\|u\|^4_E-\int_{\mathbb{R}^3}F(u)\\
	&\leq\frac{1}{2}\|u\|^2_{E}+\frac{b}{4a^2}\|u\|^4_E-c_1\int_{\mathbb{R}^3}|u|^\mu+c_2\int_{\mathbb{R}^3}|u|^2.
	\end{split}
	\end{equation*}
	Since $\mu>4$ and any norm in finite dimensional space is equivalent, one concludes that $$\lim_{\mbox{\tiny$\begin{array}{c}
			u\in E_m\\\|u\|_E\to\infty
			\end{array}$}}I(u)=-\infty$$ for any fixed $m\geq 1$. Therefore the result follows.
\end{proof}

\begin{lemma}\label{lem3.0}
	Assume $(f_1)$ and $(f_2)$ hold. Then there exist $\rho>0$ and $\alpha>0$ such that
	$$\inf_{\partial \mathcal{B}_\rho}I\geq \alpha.$$
\end{lemma}

\begin{proof}
	By $(f_1)$ and $(f_2)$, for any $\delta>0$, there exists $C_\delta>0$ such that
	\begin{equation}\label{eq3.0}
	|f(t)|\leq\delta|t|+C_\delta|t|^{p-1}
	\end{equation}
	for $t\in\mathbb{R}$.
	Then, for $u\in E$, we have
	\begin{equation}\label{eq3.00}
	\begin{split}
	I(u)&=\frac{1}{2}\|u\|^2_{E}+\frac{b}{4}\left(\int_{\mathbb{R}^3}|\nabla u|^2\right)^2-\int_{\mathbb{R}^3}F(u)\\
	&\geq\frac{1}{2}\|u\|^2_{E}-\frac{\delta}{2}\int_{\mathbb{R}^3}|u|^2-\frac{C_\delta}{p}\int_{\mathbb{R}^3}|u|^p.
	\end{split}
	\end{equation}
	By the Sobolev embedding theorem, for any $r\in[2,6]$, there exists $C(r)>0$ such that $\|u\|_r\leq C(r)\|u\|_E$. Choose $\delta$ satisfying that $C(2)\delta<\frac{1}{2}$. Then, it follows from \eqref{eq3.00} that
	\begin{equation*}
	I(u)\geq\frac{1}{4}\|u\|^2_{E}-\frac{C_\delta C(p)}{p}\|u\|_E^p.
	\end{equation*}
	Noting that $p>2$, we conclude that there exist $\rho,\alpha>0$ such that
	$\inf_{\partial \mathcal{B}_\rho}I\geq \alpha,$ as required.
\end{proof}

\subsection{Invariant subsets of descending flow}

In order to construct a descending flow guaranteeing existence of the desired invariant sets for the functional $I$, we introduce an auxiliary operator $A:E\to E$. Precisely, for any $u\in E$, we define $v=A(u)$ the unique solution to the following equation
\begin{equation}\label{eq3.1}
-\left(a+b\int_{\mathbb{R}^3}|\nabla u|^2\right)\Delta v+V(x)v=f(u), \quad v\in E.
\end{equation}
Clearly, the set of fixed points of $A$ is the same as the set of critical points of $I$.

\begin{lemma}\label{lem3.00}
	The operator $A$ is well defined and continuous.
\end{lemma}

\begin{proof}
	Let $u\in E$, and define $$J(v)=\frac{1}{2}\left(a+b\int_{\mathbb{R}^3}|\nabla u|^2\right)\int_{\mathbb{R}^3}|\nabla v|^2+\frac{1}{2}\int_{\mathbb{R}^3}V(x)v^2-\int_{\mathbb{R}^3}f(u)v,\quad v\in E.$$
	Obviously, $J\in C^1(E,\mathbb{R})$. And it is easy to check that $J$ is weakly lower semicontinuous.
	
	By \eqref{eq3.0} and the Sobolev embeddings, one has
	\begin{equation}\label{eq3.2}
	\begin{split}
	\left|\int_{\mathbb{R}^3}f(u)v\right|\leq\delta\int_{\mathbb{R}^3}|u||v|+C_\delta\int_{\mathbb{R}^3}|u|^{p-1}|v|
	&\leq\delta\|u\|_2\|v\|_2+C_\delta\|u\|_p^{p-1}\|v\|_p\\
	&\leq C\|v\|_E,
	\end{split}
	\end{equation}
	which implies
	\begin{equation*}
	\begin{split}
	J(v)&=\frac{1}{2}\left(a+b\int_{\mathbb{R}^3}|\nabla u|^2\right)\int_{\mathbb{R}^3}|\nabla v|^2+\frac{1}{2}\int_{\mathbb{R}^3}V(x)v^2-\int_{\mathbb{R}^3}f(u)v\\
	&\geq\frac{1}{2}\|v\|_E^2-C\|v\|_E\to+\infty,\quad \textrm{as}\,\,\|v\|_E\to\infty,
	\end{split}
	\end{equation*}
	where $C=C(\varepsilon,u)$ is a constant depending on $\varepsilon$ and $u$. Therefore, $J$ is coercive.
	
	By \eqref{eq3.2}, it is easy to see that $J$ is bounded below and maps bounded sets into bounded sets. Now we prove $J$ is also strictly convex. In fact,
	\begin{equation*}
	\begin{split}
	\langle J'(v)-J'(w),v-w\rangle&=\left(a+b\int_{\mathbb{R}^3}|\nabla u|^2\right)\int_{\mathbb{R}^3}|\nabla (v-w)|^2+\int_{\mathbb{R}^3}V(x)|v-w|^2\\
	&\geq\|v-w\|_E^2>0,\quad \textrm{if}\,\,v\neq w,
	\end{split}
	\end{equation*}
	where
	$$\langle J'(v),\varphi\rangle=\left(a+b\int_{\mathbb{R}^3}|\nabla u|^2\right)\int_{\mathbb{R}^3}\nabla v\cdot\nabla \varphi+\int_{\mathbb{R}^3}V(x)v\varphi-\int_{\mathbb{R}^3}f(u)\varphi,\quad \varphi\in E.$$
	Thus, by Theorems 1.5.6 and 1.5.8 in \cite{Badiale2011}, $J$ admits a unique minimizer $v=A(u)\in E$, which is the unique solution to \eqref{eq3.1}. Moreover, by \eqref{eq3.2}, $A$ maps bounded sets into bounded sets.
	
	In the following, we prove that the operator $A$ is continuous. Let $\{u_n\}\subset E$ with $u_n\to u$ in $E$ strongly. Denote $v=A(u)$ and $v_n=A(u_n)$. Then we have
	\begin{equation*}
	\begin{split}
	\|v_n-v\|_E^2
	&=b\int_{\mathbb{R}^3}|\nabla u|^2\int_{\mathbb{R}^3}\nabla v\nabla (v_n-v)-b\int_{\mathbb{R}^3}|\nabla u_n|^2\int_{\mathbb{R}^3}\nabla v_n\cdot\nabla (v_n-v)\\
	&\quad+\int_{\mathbb{R}^3}(f(u_n)-f(u))(v_n-v)\\
	&\leq b\left(\int_{\mathbb{R}^3}|\nabla u|^2-\int_{\mathbb{R}^3}|\nabla u_n|^2\right)\int_{\mathbb{R}^3}\nabla v\cdot\nabla (v_n-v)\\
	&\quad+\int_{\mathbb{R}^3}(f(u_n)-f(u))(v_n-v)\\
	&\triangleq I_1+I_2.
	\end{split}
	\end{equation*}
	By the H\"{o}lder inequality and Sobolev embedding theorem,
	\begin{equation}\label{eq3.3}
	\begin{split}
	I_1&\leq b\left|\int_{\mathbb{R}^3}|\nabla u|^2-\int_{\mathbb{R}^3}|\nabla u_n|^2\right|\|v\|_2\|v_n-v\|_2\\
	&\leq C\left|\int_{\mathbb{R}^3}|\nabla u|^2-\int_{\mathbb{R}^3}|\nabla u_n|^2\right|\|v\|_E\|v_n-v\|_E.
	\end{split}
	\end{equation}
	Now, we estimate the second term $I_2$. The proof is similar to that of \cite{Liu2015}. Let $\phi\in C^{\infty}_{0}(\mathbb{R})$ be such that $\phi(t)\in[0,1]$ for $t\in\mathbb{R}$, $\phi(t)=1$ for $\|t\|\leq1$ and $\phi(t)=0$ for $\|t\|\geq2$. Let $g_1(t)=\phi(t)f(t)$, $g_2(t)=f(t)-g_1(t)$. Then, by $(f_1)$ and $(f_2)$, there exists $C>0$ such that $|g_1(s)|\leq C|s|$ and $|g_2(s)|\leq C|s|^{p-1}$ for $s\in\mathbb{R}$. Thus,
	\begin{equation*}
	\begin{split}
	I_2&=\int_{\mathbb{R}^3}(g_1(u_n)-g_1(u))(v_n-v)+\int_{\mathbb{R}^3}(g_2(u_n)-g_2(u))(v_n-v)\\
	&\leq \left(\int_{\mathbb{R}^3}|g_1(u_n)-g_1(u)|^2\right)^{\frac{1}{2}}\left(\int_{\mathbb{R}^3}|v_n-v|^2\right)^{\frac{1}{2}}\\
	&\quad+\left(\int_{\mathbb{R}^3}|g_2(u_n)-g_2(u)|^p\right)^{\frac{p-1}{p}}\left(\int_{\mathbb{R}^3}|v_n-v|^p\right)^{\frac{1}{p}}\\
	&\leq C\|v_n-v\|_E\left[\left(\int_{\mathbb{R}^3}|g_1(u_n)-g_1(u)|^2\right)^{\frac{1}{2}}+\left(\int_{\mathbb{R}^3}|g_2(u_n)-g_2(u)|^p\right)^{\frac{p-1}{p}}\right],
	\end{split}
	\end{equation*}
	which, jointly with \eqref{eq3.3}, implies
	\begin{equation*}
	\begin{split}
	\|v_n-v\|_E&\leq C\Bigg[\left|\int_{\mathbb{R}^3}|\nabla u|^2-\int_{\mathbb{R}^3}|\nabla u_n|^2\right|\|v\|_E+\left(\int_{\mathbb{R}^3}|g_1(u_n)-g_1(u)|^2\right)^{\frac{1}{2}}\\
	&\quad+\left(\int_{\mathbb{R}^3}|g_2(u_n)-g_2(u)|^p\right)^{\frac{p-1}{p}}\Bigg].
	\end{split}
	\end{equation*}
	Therefore, noting that $u_n\to u$ in $E$ and by the dominated convergence theorem, one has $\|v_n-v\|_E\to 0$ as $n\to\infty$. This proof is completed.
\end{proof}

Now we summarize some properties of the operator $A$ which are useful to study our problem.
\begin{lemma}\label{lem3.1}
	\
	
	\begin{itemize}
		\item [$\rm{(i)}$] $\langle I'(u),u-A(u)\rangle\geq\|u-A(u)\|_E^2$ for all $u\in E$.
		\item [$\rm{(ii)}$] $\|I'(u)\|\leq\|u-A(u)\|_E(1+C\|u\|_E^2)$ for some $C>0$ and all $u\in E$ .
		\item [$\rm{(iii)}$] For $M>0$ and $\alpha>0$, there exists $\beta>0$ such that $\|u-A(u)\|_E\geq \beta$ for any $u\in E$ with $|I(u)|\leq M$ and $\|I'(u)\|\geq \alpha$.
		\item [$\rm{(iv)}$] If $f$ is odd, then so is $A$.
	\end{itemize}
\end{lemma}

\begin{proof}
	(i) Noting that $A(u)$ is the solution to \eqref{eq3.1}, for $u\in E$, it is easy to see that
	\begin{equation}\label{eq3.4}
	\begin{split}
	\langle I'(u),u-A(u)\rangle&=\left(a+b\int_{\mathbb{R}^3}|\nabla u|^2\right)\int_{\mathbb{R}^3}|\nabla (u-A(u))|^2\\
	&\quad+\int_{\mathbb{R}^3}V(x)|u-A(u)|^2\\
	&\geq\|u-A(u)\|_E^2.
	\end{split}
	\end{equation}
	
	(ii) For any $\varphi\in E$, By the H\"{o}lder inequality, one has
	\begin{equation*}
	\begin{split}
	\langle I'(u),\varphi\rangle&=(u-A(u),\varphi)_E+b\int_{\mathbb{R}^3}|\nabla u|^2\int_{\mathbb{R}^3}\nabla (u-A(u))\cdot\nabla \varphi\\
	&\leq\|u-A(u)\|_E\|\varphi\|_E+C\|u\|_E^2\|u-A(u)\|_E\|\varphi\|_E,
	\end{split}
	\end{equation*}
	which implies $\|I'(u)\|\leq\|u-A(u)\|_E(1+C\|u\|_E^2)$ for all $u\in E$, here $C=b/a^2>0$ is a constant.
	
	(iii) Since
	\begin{equation*}
	\langle I'(u),u\rangle=(u-A(u),u)_E+b\int_{\mathbb{R}^3}|\nabla u|^2\int_{\mathbb{R}^3}\nabla (u-A(u))\nabla u,
	\end{equation*}
	it follows from $(f_3)$ that
	\begin{equation*}
	\begin{split}
	&I(u)-\frac{1}{\mu}(u-A(u),u)_E\\
	&\quad=\left(\frac{1}{2}-\frac{1}{\mu}\right)\|u\|_E^2+\left(\frac{1}{4}-\frac{1}{\mu}\right)b\left(\int_{\mathbb{R}^3}|\nabla u|^2\right)^2\\
	&\qquad+\frac{b}{\mu}\int_{\mathbb{R}^3}|\nabla u|^2\int_{\mathbb{R}^3}\nabla (u-A(u))\nabla u+\int_{\mathbb{R}^3}\left(\frac{1}{\mu}f(u)u-F(u)\right)\\
	&\quad\geq\left(\frac{1}{2}-\frac{1}{\mu}\right)\|u\|_E^2+\left(\frac{1}{4}-\frac{1}{\mu}\right)b\left(\int_{\mathbb{R}^3}|\nabla u|^2\right)^2\\
	&\quad+\frac{b}{\mu}\int_{\mathbb{R}^3}|\nabla u|^2\int_{\mathbb{R}^3}\nabla u\cdot\nabla (u-A(u)).
	\end{split}
	\end{equation*}
	Consequently,
	\begin{equation}\label{eq3.5}
	\begin{split}
	\|u\|_E^2&+\left(\int_{\mathbb{R}^3}|\nabla u|^2\right)^2\\&\quad\leq C\Bigg(|I(u)|+\|u\|_E\|u-A(u)\|_E
	+\int_{\mathbb{R}^3}|\nabla u|^2\left|\int_{\mathbb{R}^3}\nabla u\cdot\nabla (u-A(u))\right|\Bigg).
	\end{split}
	\end{equation}
	By the H\"{o}lder inequality and Young inequality, for any $\varepsilon>0$,
	\begin{equation*}
	\begin{split}
	\int_{\mathbb{R}^3}|\nabla u|^2&\left|\int_{\mathbb{R}^3}\nabla u\cdot\nabla (u-A(u))\right|\\&\leq \int_{\mathbb{R}^3}|\nabla u|^2\left(\int_{\mathbb{R}^3}|\nabla u|^2\right)^{\frac{1}{2}}\left(\int_{\mathbb{R}^3}|\nabla (u-A(u))|^2\right)^{\frac{1}{2}}\\
	&\leq \varepsilon\left(\int_{\mathbb{R}^3}|\nabla u|^2\right)^2+C(\varepsilon)\|u\|_E^2\|u-A(u)\|_E^2.
	\end{split}
	\end{equation*}
	Then, for $\varepsilon$ small enough, from \eqref{eq3.5}, we have
	\begin{equation}\label{eq3.6}
	\begin{split}
	\|u\|_E^2\leq C_1(|I(u)|+\|u\|_E\|u-A(u)\|_E
	+\|u\|_E^2\|u-A(u)\|_E^2).
	\end{split}
	\end{equation}
	Arguing indirectly, suppose that there exists $\{u_n\}\subset E$ with $|I(u_n)|\leq M$ and $\|I'(u_n)\|\geq \alpha$ such that $\|u_n-A(u_n)\|_E\to0$ as $n\to\infty$. Then it follows from \eqref{eq3.6} that $\{\|u_n\|_E\}$ is bounded. Thus, by (ii), one concludes that $\|I'(u_n)\|\to0$ as $n\to\infty$, which is a contradiction.
	
	The conclusion (iv) is obviously, and the proof is completely.
\end{proof}

Here and in the sequel, define the convex cones $$P^+:=\{u\in E:u\geq0\}\quad \textrm{and}\quad P^-:=\{u\in E:u\leq0\}.$$
For $\varepsilon>0$, we denote $$P_\varepsilon^+:=\{u\in E:\textrm{dist}(u,P^+)<\varepsilon\}\quad \textrm{and}\quad P_\varepsilon^-:=\{u\in E:\textrm{dist}(u,P^-)<\varepsilon\},$$
where $\textrm{dist}(u,P^\pm)=\inf_{v\in P^{\pm}}\|u-v\|_{E}$. Obviously, $P_\varepsilon^-=-P_{\varepsilon}^+.$ Let $W:=P_{\varepsilon}^+\cup P_{\varepsilon}^-.$ It is easy to see that $W$ is an open and symmetric subset of $E$ and $Q:=E\setminus W$ contains only sign-changing functions. The following result shows that for $\varepsilon$ small, all sign-changing solutions to \eqref{eq1} are contained in $Q$.

\begin{lemma}\label{lem3.2}
	There exists $\varepsilon_0>0$ such that for any $\varepsilon\in(0,\varepsilon_0]$, there holds
	$$A(\overline{P}^\pm_\varepsilon)\subset P^\pm_\varepsilon.$$ Moreover, every nontrivial solutions $u\in P^+_\varepsilon$ and $u\in P^-_\varepsilon$ of \eqref{eq1} are positive and negative, respectively.
\end{lemma}

\begin{proof}
	We only prove the case $A(\overline{P}^+_\varepsilon)\subset P^+_\varepsilon$, because the other case can be obtained similarly. We write $u\in E$ as $u=u^-+u^+$, where $u^+=\max\{u,0\}$ and $u^-=\min\{u,0\}$. For $u\in E$, denote $v=A(u)$. By the Sobolev embedding theorem, for any $r\in[2,6]$, there exists $C(r)>0$ such that
	\begin{equation*}
	\begin{split}
	\|u^-\|_r=\inf_{\phi\in P^{+}}\|u-\phi\|_r\leq C(r)\inf_{\phi\in P^{+}}\|u-\phi\|_E=C(r)\textrm{dist}(u,P^+).
	\end{split}
	\end{equation*}
	Then, noting that $\textrm{dist}(v,P^+)\leq\|v^-\|_E$, by \eqref{eq3.0} and H\"{o}lder inequality, we have
	\begin{equation*}
	\begin{split}
	\textrm{dist}(v,P^+)\|v^-\|_E\leq \|v^-\|_E^2&=(v,v^-)_E\\
	&=\int_{\mathbb{R}^3}f(u)v^--b\int_{\mathbb{R}^3}|\nabla u|^2\int_{\mathbb{R}^3}\nabla v\cdot\nabla v^-\\
	&\leq\int_{\mathbb{R}^3}f(u)v^-\leq\int_{\mathbb{R}^3}f(u^-)v^-\\
	&\leq\int_{\mathbb{R}^3}(\delta|u^-|+C_\delta|u^-|^{p-1})|v^-|\\
	&\leq \delta\|u^-\|_2\|v^-\|_2+C_\delta\|u^-\|_p^{p-1}\|v^-\|_p\\
	&\leq C(\delta\textrm{dist}(u,P^+)+C_\delta\textrm{dist}(u,P^+)^{p-1})\|v^-\|_E,
	\end{split}
	\end{equation*}
	which implies that $$\textrm{dist}(v,P^+)\leq C(\delta\textrm{dist}(u,P^+)+C_\delta\textrm{dist}(u,P^+)^{p-1}).$$
	Letting $\delta=\frac{1}{4C}$ and taking $\varepsilon_0\in(0,\frac{1}{(4CC_\delta)^{1/(p-2)}})$, for any $\varepsilon\in (0,\varepsilon_0]$, one has
	$$\textrm{dist}(v,P^+)\leq\frac{1}{2}\textrm{dist}(u,P^+)<\varepsilon, \quad \textrm{for\ any\ } u\in \overline{P}_\varepsilon^+.$$
	This implies that $A(\overline{P}^+_\varepsilon)\subset P^+_\varepsilon$. If there exists $u\in P_\varepsilon^+$ such that $A(u)=u,$ then $\textrm{dist}(u,P^+)=0$, i.e. $u\in P^+$. Moreover, if $u\not\equiv0$, by the maximum principle, $u>0$ in $\mathbb{R}^3$.
\end{proof}

Note that the operator $A$ is merely continuous. Denote the set of critical points of $I$ by $K$. In order to construct a descending flow for $I$, we need to construct a locally Lipschitz continuous operator on $E_0:= E\setminus K$ which inherits the main properties of $A$. Together with Lemma \ref{lem3.2} and following the same way as in the proof of Lemma 2.1 in \cite{Bartsch2005}, we have the following results.
\begin{lemma}\label{lem3.3}
	There exists a locally Lipschitz continuous operator $B: E_0\to E$ with the following properties:
	\begin{itemize}
		\item [$\rm{(1)}$] $B(\overline{P}^\pm_\varepsilon)\subset P^\pm_\varepsilon$ for all $\varepsilon\in(0,\varepsilon_0]$;
		\item [$\rm{(2)}$] $\frac{1}{2}\|u-B(u)\|_E\leq\|u-A(u)\|_E\leq2\|u-B(u)\|_E$ for all $u\in E_0$;
		\item [$\rm{(3)}$] $\langle I'(u),u-B(u)\rangle\geq\frac{1}{2}\|u-A(u)\|_E^2$ for all $u\in E_0$;
		\item [$\rm{(4)}$] If $f$ is odd, then so is $B$.
	\end{itemize}
\end{lemma}

Set $H(u)=u-B(u)$. For $u\in E_0$, consider the following initial value problems:
\begin{equation}\label{eq3.8}
\left\{\begin{array}{lcl}\frac{d}{dt}\sigma(t,u)=-H(\sigma(t,u)),\\
\sigma(0,u)=u.
\end{array}\right.
\end{equation}
Invoking the locally Lipschitz continuity of $B$, the existence and uniqueness theory of ODE implies that problem \eqref{eq3.8} has a unique solution, denoted by $\sigma(t,u)$ with maximal interval of existence $[0,T(u))$. By Lemma \ref{lem3.3}(3), it is easy to check that $I(\sigma(t,u))$ is strictly decreasing in $[0,T(u))$. Moreover, with the help of Lemma \ref{lem3.3}, using similar arguments to the proof of Lemma 3.2 in \cite{Bartsch2005}, we have the following result.

\begin{lemma}\label{lem3.4}
	For any $\varepsilon\in(0,\varepsilon_0]$ and $u\in \overline{P}^\pm_\varepsilon\setminus K$, $\sigma(t,u)\in P^\pm_\varepsilon$ for all $t\in(0,T(u))$. Here $\varepsilon_0$ is given by Lemma \ref{lem3.2}.
\end{lemma}

\begin{remark}
	since $\partial P^\pm_\varepsilon \cap K=\emptyset$ by Lemma \ref{lem3.2}, the above lemma implies that $\sigma(t,u)\in P^\pm_\varepsilon$ for all $t\in(0,T(u))$ and $u\in \partial P^\pm_\varepsilon$. In the following, we may choose an $\varepsilon>0$ sufficiently small such that $W=P^+_{\varepsilon}\cup P^-_{\varepsilon}$ is an invariant set with respect to $\sigma$.
\end{remark}

In order to construct nodal solutions by using the combination of invariant sets method and minimax method, we need a deformation lemma in the presence of invariant sets. In fact, we have the following result.

\begin{lemma}\label{lem3.5}
	Let $N$ be an open symmetric neighborhood of $K_{c}\setminus W$. Then there exists an $\varepsilon_0>0$, such that for any $0<\varepsilon<\varepsilon_0$, there exists $\eta\in C([0,1]\times E,E)$ satisfying:
	\begin{itemize}
		\item [$(\rm{i})$] $\eta(t,u)=u$ for $t=0$ or $u\notin I^{-1}([c-\varepsilon_0,c+\varepsilon_0]);$
		\item [$(\rm{ii})$] $\eta(1,I^{c+\varepsilon}\cup W\setminus N)\subset I^{c-\varepsilon}\cup W$ and $\eta(1,I^{c+\varepsilon}\cup W)\subset I^{c-\varepsilon}\cup W$ if $K_{c}\setminus W=\emptyset$;
		\item [$(\rm{iii})$] $\eta(t,-u)=-\eta(t,u)$ for all $(t,u)\in[0,1]\times E$;
		\item [$(\rm{iv})$] $\eta(t,\overline{P}^\pm_\varepsilon)\subset \overline{P}^\pm_\varepsilon$ for any $t\in[0,1].$
	\end{itemize}
\end{lemma}

\begin{proof}
	The proof is similar to that of Lemma 5.1 in \cite{Liu2005}. We state the proof here for the readers convenience. For $G\subset E$ and $a>0$, we define $N_a(G):=\{u\in E:\textrm{dist}(u,G)<a\}$.
	Set $K_c^1:=K_c\cap W$ and $K_c^2:=K_c\setminus W$.  Due to the $\rm{(PS)}$ condition, $K_c^1$ is compact and therefore $d:=\textrm{dist}(K_c^1,Q)>0$. Choose $\delta\in(0,\frac{d}{2})$ small enough such that $N_{3\delta}(K_c^2)\subset N$. Obviously, $N_\delta(K_c^1)\subset W$. Since $I$ satisfies the $\rm{(PS)}$ condition, there exists $\varepsilon_0,\alpha>0$ such that
	$$\|I'(u)\|\geq \alpha,\quad\textrm{for\ } u\in I^{-1}([c-\varepsilon_0,c+\varepsilon_0])\setminus N_\delta(K_c).$$
	Then, from Lemmas \ref{lem3.1} and \ref{lem3.3}, there exists $\beta>0$ such that
	\begin{equation}\label{eq3.9}
	\|u-B(u)\|_E\geq \beta,\quad\textrm{for\ } u\in I^{-1}([c-\varepsilon_0,c+\varepsilon_0])\setminus N_\delta(K_c).
	\end{equation}
	Without loss of generality, we may assume that $\varepsilon_0\leq\frac{\beta\delta}{16}$.
	Define $$E_1:=I^{-1}([c-\varepsilon_0,c+\varepsilon_0])\setminus N_\delta(K_c^2),$$ and for any fixed $\varepsilon\in(0,\varepsilon_0)$, we set $$E_2:=I^{-1}([c-\varepsilon,c+\varepsilon])\setminus N_{2\delta}(K_c^2).$$ Let
	$$\psi(u)=\frac{\textrm{dist}(u,E\setminus E_1)}{\textrm{dist}(u,E\setminus E_1)+\textrm{dist}(u,E_2)}.$$
	Recall that $H(u)=u-B(u)$ for $u\in E_0=E\setminus K$. By Lemma \ref{lem3.3}, $\psi(\cdot)H(\cdot)$ is locally Lipschitz continuous on $E$. Consider the following initial value problem
	\begin{equation*}
	\frac{d}{dt}\xi(t,u)=-\frac{\psi(\xi(t,u))H(\xi(t,u))}{\|H(\xi(t,u))\|_E},\quad
	\xi(0,u)=u.
	\end{equation*}
	Then $\xi(t,u)$ is well-defined and continuous on $\mathbb{R}^+\times E$.
	
	Define $\eta(t,u):=\xi(\delta t,u)$. It suffices to check $(\rm{ii})$, because $(\rm{i})$ is obviously and $(\rm{iii})$-$(\rm{iv})$ are easily checked by Lemma \ref{lem3.3}(4) and \ref{lem3.4}. For $(\rm{ii})$, arguing indirectly, we suppose that $\eta(1,u)\notin I^{c-\varepsilon}\cup W$ for some $u\in I^{c+\varepsilon}\cup W\setminus N$. Then $\eta(t,u)\notin I^{c-\varepsilon}\cup W$ for all $t\in[0,1]$ since $I^{c-\varepsilon}\cup W$ is an invariant subset. As a result, $\eta(t,u)\notin N_\delta(K_c^1)$ for all $t\in[0,1]$ since $N_\delta(K_c^1)\subset W$. Noting that $$\|\eta(t,u)-u\|_E=\|\xi(\delta t,u)-u\|_E=\left\|\int_0^{\delta t}\xi'(s,u)ds\right\|_E\leq \delta, \quad\textrm{for\ any\ }t\in[0,1],$$
	we have $\eta(t,u)\notin N_{2\delta}(K_c^2)$ for all $t\in[0,1]$ due to $u\notin N_{3\delta}(K_c^2)$. Therefore, for all $t\in[0,1]$, $$\eta(t,u)\in I^{-1}([c-\varepsilon,c+\varepsilon])\setminus\left(N_{2\delta}(K_c^2)\cup N_\delta(K_c^1)\right).$$
	As a consequence, for $t\in[0,1]$, $$\psi(\eta(t,u))=1\quad\textrm{and}\quad \|u-B(u)\|_E\geq \beta.$$
	Hence, by Lemmas \ref{lem3.3}(2)-(3) and \eqref{eq3.9}, we have
	\begin{equation*}
	\begin{split}
	I(\eta(t,u))=I(u)+\int_0^1\frac{d}{ds}I(\eta(s,u))ds&\leq c+\varepsilon-\int_0^1\frac{\delta}{8}\|\eta(s,u)-B(\eta(s,u))\|_Eds\\
	&\leq c+\varepsilon-\frac{\beta\delta}{8}\leq c+\varepsilon-2\varepsilon_0<c-\varepsilon,
	\end{split}
	\end{equation*}
	a contradiction. The proof is completed.
\end{proof}

\subsection{The proof of Theorem \ref{thm1.1}}

Now we are in a position to prove Theorem \ref{thm1.1}.

\begin{proof}[Proof of Theorem \ref{thm1.1}] We adopt some techniques in the proof of Theorem 1.1 in \cite{Sun2014} (see also \cite{Liu2005,Rabinowitz1986}). We divide our proof into three steps.
	
	{\bf Step 1.} We define a minimax value $c_{k}$ for the functional $I(u)$ with $k=2,3,\cdots$.
	
	Define
	$$G_m:=\{h\in C(\mathcal{B}_{R_m}\cap E_m,E): h \textrm{\ is\ odd\ and\ } h=\text{id} \textrm{\ on\ }\partial\mathcal{B}_{R_m}\cap E_m\},$$
	where $R_m>0$ is given by Lemma \ref{lem3}. It is easy to see that $G_m\neq\emptyset$, because $\text{id}\in G_m$ for all $m\in\mathbb{N}$. For $k\geq2$, we set
	$$\Gamma_k:=\{h(\mathcal{B}_{R_m}\cap E_m\setminus Y):h\in G_m, m\geq k, Y=-Y \textrm{\ is\ open\ and\ } \gamma(Y)\leq m-k\},$$
	where $\gamma(M)$ denote the Krasnoselskii's genus of the set $M$ (cf. \cite{Struwe2000}).
	As Proposition 9.18 in \cite{Rabinowitz1986}, $\Gamma_k$ possess the following properties:
	\begin{itemize}
		\item [$(\rm{1^{\circ}})$] $\Gamma_k\neq\emptyset$ and $\Gamma_{k+1}\subset \Gamma_k$ for all $k\geq2$.
		\item [$(\rm{2^{\circ}})$] If $\phi\in C(E,E)$ is odd and $\phi=\text{id}$ on $\partial\mathcal{B}_{R_m}\cap E_m$, then $\phi(A)\in \Gamma_k$ if $A\in\Gamma_k$ for all $k\geq2$.
		\item [$(\rm{3^{\circ}})$] If $A\in\Gamma_k$, $Z=-Z$ is open and $\gamma(Z)\leq s<k$ and $k-s\geq2$, then $A\setminus Z\in \Gamma_{k-s}.$
	\end{itemize}
	
	Now, for $k\geq2$, we claim that for any $A\in \Gamma_k$, $A\cap Q\neq\emptyset$. Consider the attracting domain of $0$ in $E$:
	$$\mathcal{D}:=\{u\in E:\sigma(t,u)\to0, \textrm{\ as\ } t\to\infty\}.$$
	Since $0$ is a local minimum of $I$ and by the continuous dependence of ODE on initial data, $\mathcal{D}$ is open. Moreover, $\partial\mathcal{D}$ is an invariant set and $\overline{P^+_\varepsilon}\cap\overline{P^-_\varepsilon}\subset \mathcal{D}.$ Similar to Lemma 3.4 in \cite{Bartsch2004}, we have $$I(u)>0  \text{ for every } u\in\overline{P^+_\varepsilon}\cap\overline{P_\varepsilon^-}\setminus\{0\}.$$
	Given $k\geq2$, let $A\in\Gamma_k$, that is
	$$A=h(\mathcal{B}_{R_m}\cap E_m\setminus Y)$$ with $\gamma(Y)\leq m-k$. Define $$\mathcal{O}:=\{u\in \mathcal{B}_{R_m}\cap E_m:h(u)\in \mathcal{D}\}.$$ Obviously, $\mathcal{O}$ is a bounded open symmetric set with $0\in\mathcal{O}$ and $\overline{\mathcal{O}}\subset\mathcal{B}_{R_m}\cap E_m$. Therefore, by the Borsuk-Ulam theorem, $\gamma(\partial\mathcal{O})=m$. Moreover, by the continuity of $h$, $h(\partial\mathcal{O})\subset\partial\mathcal{D}$. Consequently, $$h(\partial\mathcal{O}\setminus Y)\subset A\cap\partial\mathcal{D}.$$ Thus, by using the ``monotone, sub-additive and supervariant" property of the genus (cf. Proposition 5.4 in \cite{Struwe2000}), we have $$\gamma(A\cap\partial\mathcal{D})\geq\gamma(h(\partial\mathcal{O}\setminus Y))\geq\gamma(\partial\mathcal{O}\setminus Y)\geq\gamma(\partial\mathcal{O})-\gamma(Y)\geq k.$$ Noting that $P^+_\varepsilon\cap P^-_\varepsilon\cap \partial\mathcal{D}=\emptyset$, one has $\gamma(W\cap \partial\mathcal{D})\leq1.$ Hence, for $k\geq2,$ we conclude that $$\gamma(A\cap Q\cap\partial\mathcal{D})\geq\gamma(A\cap\partial\mathcal{D})-\gamma(W\cap \partial\mathcal{D})\geq k-1\geq1,$$
	which implies
	\begin{equation}\label{eq32.4}
	A\cap Q\cap\partial\mathcal{D}\neq\emptyset.
	\end{equation} Then, it follows that $A\cap Q\neq\emptyset$ for any $A\in \Gamma_k$ with $k\geq2$. Thus, we finish the proof of the claim. Hence, for $k=2,3,\cdots$, we can define a minimax value $c_{k}$ by
	$$c_{k}:=\inf_{A\in \Gamma_k}\sup_{A\cap Q}I.$$ Choosing $\rho$ given in Lemma \ref{lem3.0} small enough if necessary, we have $\partial\mathcal{B}_\rho\subset \mathcal{D}$. Then by \eqref{eq32.4}, for any $A\in \Gamma_k$, one has
	$$\sup_{A\cap Q}I\geq\inf_{\partial\mathcal{D}}I\geq\inf_{\partial\mathcal{B}_\rho}I\geq\alpha>0$$ by Lemma \ref{lem3.0}. As a consequence, $c_{k}\geq \alpha>0.$ Moreover, by $(\rm{1^{\circ}})$, $c_{k+1}\geq c_k$ for any $k\geq2$.
	
	{\bf Step 2.} We show that for all $k\geq2$, there exists a sign-changing critical point $u_{k}$ such that $I(u_{k})=c_{k}$ i.e.,
	\begin{equation}\label{eq32.5}
	K_{c_{k}}\cap Q\neq\emptyset.
	\end{equation}
	
	To prove \eqref{eq32.5}, arguing by contradiction, we suppose $K_{c_{k}}\cap Q=\emptyset.$
	By Lemma \ref{lem3.5}, there exist $\varepsilon>0$ and a map $\eta\in C([0,1]\times E,E)$ such that $\eta(1,\cdot)$ is odd, $\eta(1,u)=u$ for $u\in I^{c_{k}-2\varepsilon}$ and
	\begin{equation}\label{eq32.6}
	\eta(1,I^{c_{k}+\varepsilon}\cup W)\subset I^{c_{k}-\varepsilon}\cup W.
	\end{equation}
	It follows from the definition of $c_{k}$ that there exists $A\in \Gamma_k$ such that $$\sup_{A\cap Q}I\leq c_{k}+\varepsilon.$$ Set $B=\eta(1,A)$. Then, by \eqref{eq32.6}, we have $$\sup_{B\cap Q}I\leq c_{k}-\varepsilon.$$ Noting that $B\in \Gamma_k$ by Lemma \ref{lem3} and $(\rm{2^{\circ}})$ above, one concludes that $c_{k}\leq c_{k}-\varepsilon,$ a contradiction.
	
	{\bf Step 3.} Finally, we shall prove that $c_{k}\to\infty$, as $k\to\infty.$ This implies that $I(u)$ has infinitely many sign-changing critical points.
	
	Arguing indirectly, we assume $c_{k}\to c<\infty,$ as $k\to\infty.$ Owing to the $\rm(PS)$ condition, it follows that $K_{c}$ is nonempty and compact. Moreover, we have $$K^2_{c}:=K_{c}\cap Q\neq\emptyset.$$ In fact, suppose $\{u_{k}\}$ is a sequence of sign-changing solutions to \eqref{eq1} with $I(u_{k})=c_{k}$. Then, $\langle I'(u_k),u_k^\pm\rangle=0$ and therefore $$\|u_{k}^{\pm}\|_{E}^2+b\int_{\mathbb{R}^3}|\nabla u_k|^2\int_{\mathbb{R}^3}|\nabla u^\pm_k|^2=\int_{\mathbb{R}^3}f(u_k^{\pm})u_k^\pm.$$
	Thus by \eqref{eq3.0} and Sobolev embedding theorem, we have $\|u_{k}^{\pm}\|_{E}\geq \delta_0>0,$ where $\delta_0$ is a constant independent of $k$. Noting that $I$ satisfies the $\rm(PS)$ condition, passing to a subsequence if necessary, there exists $u\in K_{c}$ such that $u_{k}\to u$. Then, the above inequality implies that $u$ is still sign-changing and hence $K^2_{c}\neq\emptyset.$
	
	Suppose $\gamma(K^2_{c})=\tau$. Since $0\notin K^2_{c}$ and $K^2_{c}$ is compact, by the ``continuous" property of the genus, there exists a open neighborhood $N$ in $E$ with $K^2_{c}\subset N$ such that $\gamma(N)=\tau.$
	From Lemma \ref{lem3.5}, there exist $\varepsilon>0$ and a map $\eta\in C([0,1]\times E,E)$ such that $\eta(1,\cdot)$ is odd, $\eta(1,u)=u$ for $u\in I^{c-2\varepsilon}$ and
	\begin{equation}\label{eq32.7}
	\eta(1,I^{c+\varepsilon}\cup W\setminus N)\subset I^{c-\varepsilon}\cup W.
	\end{equation}
	Since $c_{k}\to c$ as $k\to\infty$, we can choose $k$ sufficiently large, such that
	\begin{equation}\label{eq32.8}
	c_{k}\geq c-\frac{1}{2}\varepsilon.
	\end{equation}
	Note that $c_{k+\tau}\geq c_{k}$. By the definition of $c_{k+\tau}$, there exists $A\in\Gamma_{k+\tau}$, i.e., $$A=h(\mathcal{B}_R\cap E_m\setminus Y),$$ where $h\in G_m$, $m\geq k+\tau$, $\gamma(Y)\leq m-(k+\tau),$ such that
	$$I(u)\leq c_{k+\tau}+\frac{1}{4}\varepsilon<c+\varepsilon,\quad \textrm{for\ any \ } u\in A\cap Q.$$
	Therefore $A\subset I^{c+\varepsilon}\cup W.$ Then, from \eqref{eq32.7}, we have
	\begin{equation}\label{eq32.9}
	\eta(1,A\setminus N)\subset I^{c-\varepsilon}\cup W.
	\end{equation}
	Set $Y_1=Y\cup h^{-1}(N).$ Clearly, $Y_1$ is symmetric and open, and $$\gamma(Y_1)\leq \gamma(Y)+\gamma(h^{-1}(N))\leq m-(k+\tau)+\tau=m-k.$$ Then, by $(\rm{2^{\circ}})$ and $(\rm{3^{\circ}})$ above, one concludes that $$\widetilde{A}:=\eta(1,h(\mathcal{B}_R\cap E_m\setminus Y_1))\in\Gamma_k.$$ As a result, by \eqref{eq32.9}
	$$c_{k}\leq \sup_{\widetilde{A}\cap Q}I\leq\sup_{\eta(1,A\setminus N)\cap Q}I\leq c-\varepsilon.$$
	This is a contradiction to \eqref{eq32.8} and hence the proof is completed.
\end{proof}

\section{Proof of Theorem 1.2}\label{sec4}

Fix a number $r\in(\max\{4,p\},6)$. As in \cite{Liu2015}, we introduce a family of functional defined by
\begin{equation}\label{eq4.000}
I_\lambda(u)=\frac{a}{2}\int_{\mathbb{R}^3}|\nabla u|^2+\frac{b}{4}\left(\int_{\mathbb{R}^3}|\nabla u|^2\right)^2+\frac{1}{2}\int_{\mathbb{R}^3}V(x)|u|^2-\int_{\mathbb{R}^3}F(u)-\frac{\lambda}{r}\int_{\mathbb{R}^3}|u|^r
\end{equation}
for $\lambda\in(0,1]$. It is standard to show that $I_\lambda\in C^1(E,\mathbb{R})$.

In order to study our problem, we give some preliminary results. Firstly, as Lemma 2.1 in \cite{Li2014}, we have the following Pohozaev type identity.
\begin{lemma}\label{lem4.0}
	Assume $(V_0)$--$(V_1)$ and $(f_1)$--$(f_3)$ hold. Let $u$ be a critical point of $I_\lambda$ in $E$, then
	\begin{equation*}
	\begin{split}
	\frac{a}{2}\int_{\mathbb{R}^3}|\nabla u|^2+\frac{3}{2}\int_{\mathbb{R}^3}V(x)|u|^2&+\frac{1}{2}\int_{\mathbb{R}^3}(DV(x),x)|u|^2+\frac{b}{2}\left(\int_{\mathbb{R}^3}|\nabla u|^2\right)^2\\
	&-3\int_{\mathbb{R}^3}F(u)-\frac{3\lambda}{r}\int_{\mathbb{R}^3}|u|^r=0.
	\end{split}
	\end{equation*}
\end{lemma}

\begin{lemma}\label{lem4.1}
	For $\lambda\in(0,1]$, $I_\lambda$ satisfies the $\rm{(PS)}$ condition.
\end{lemma}

\begin{proof}
	Here we adopt a technique in the proof of Lemma 4.2 in \cite{Liu2015}. Assume that there exist $\{u_n\}\subset E$ and $c\in\mathbb{R}$ such that $I_\lambda(u_n)\to c$ and $I'_\lambda(u_n)\to0$ as $n\to\infty$. Choose a number $\gamma\in(4,r)$. For $u\in E$, we have
	\begin{equation*}
	\begin{split}
	I_\lambda(u_n)-\frac{1}{\gamma}\langle I'_\lambda(u_n),u_n\rangle&=\left(\frac{1}{2}-\frac{1}{\gamma}\right)\|u_n\|_E^2+\left(\frac{1}{4}-\frac{1}{\gamma}\right)b\left(\int_{\mathbb{R}^3}|\nabla u_n|^2\right)^2\\
	&\quad+\int_{\mathbb{R}^3}\left(\frac{1}{\gamma}f(u_n)u_n-F(u_n)\right)+\left(\frac{1}{\gamma}-\frac{1}{r}\right)\lambda\|u_n\|_r^r.
	\end{split}
	\end{equation*}
	Then, by $(f_1)$ and $(f_2)$, one sees that
	$$\|u_n\|_E^2+\lambda\|u_n\|_r^r\leq C_1(|I_\lambda(u_n)|+\|u_n\|_E\|I'_\lambda(u_n)\|+\|u_n\|_p^p).$$
	Thus, by the Young inequality, for large $n$,
	\begin{equation}\label{eq4.01}
	\|u_n\|_E^2+\lambda\|u_n\|_r^r\leq C_2(1+\|u_n\|_p^p).
	\end{equation}
	Now we show that $\{u_n\}$ is bounded in $E$. Arguing indirectly, we assume that $\|u_n\|_E\to\infty$ as $n\to\infty$. Let $w_n=u_n/\|u_n\|_E$. Then $\|w_n\|=1$ and thus up to a subsequence if necessary, there exists $w\in E$ such that $w_n\rightharpoonup w$ in $E$ and $w_n\to w$ in $L^s(\mathbb{R}^3)$ for any $s\in[2,6)$. Since $2<p<r$, by \eqref{eq4.01}, it is easy to see that $$\|w\|_r^r\leq0,$$
	which implies $w=0$. On the other hand, from \eqref{eq4.01}, there exists $C(\lambda)>0$ such that for large $n$,
	$$\|u_n\|_2^2+\|u_n\|_r^r\leq C(\lambda)\|u_n\|_p^p\leq C(\lambda)\|u_n\|_2^{tp}\|u_n\|_r^{(1-t)p},$$
	here we have used the interpolation inequality and $t\in (0,1)$ satisfying $\frac{1}{p}=\frac{t}{2}+\frac{1-t}{r}$. As a consequence, there exist $C_1(\lambda),C_2(\lambda)>0$ such that, for large $n$,
	$$C_1(\lambda)\|u_n\|_2^{\frac{2}{r}}\leq\|u_n\|_r\leq C_2(\lambda)\|u_n\|_2^{\frac{2}{r}}.$$
	Therefore $\|u_n\|_p^p\leq C_3(\lambda)\|u_n\|_2^2$ and hence, by \eqref{eq4.01}, for large $n$, we have
	\begin{equation*}
	\|u_n\|_E^2\leq C_4(\lambda)\|u_n\|_2^2,
	\end{equation*}
	which implies that $\|w_n\|_2^2\geq(C_4(\lambda))^{-1}$. Noting that $w_n\to w$ in $L^2(\mathbb{R}^3)$, we see that $\|w\|_2^2\geq(C_4(\lambda))^{-1}$. This contradicts to $w=0$ and therefore $\{u_n\}$ is bounded in $E$.
	
	Then, up to a subsequence, we can assume that $u_n\rightharpoonup u$ in $E$ as $n\to\infty$. By Lemma \ref{lem2}, we see that
	\begin{equation}\label{eq4.5}
	u_n\rightarrow u \,\,\,\, \textrm{\ in\ } L^s(\mathbb{R}^3),\quad \textrm{for\ any\ } s\in[2,6).
	\end{equation}
	Note that
	\begin{equation*}
	\begin{split}
	\langle I'_\lambda(u_n)-I'_\lambda(u),u_n-u\rangle&=\|u_n-u\|_E^2+b\int_{\mathbb{R}^3}|\nabla u_n|^2\int_{\mathbb{R}^3}\nabla u_n\cdot\nabla(u_n-u)\\
	&\quad+b\int_{\mathbb{R}^3}|\nabla u|^2\int_{\mathbb{R}^3}\nabla u\nabla(u_n-u)\\
	&\quad-\int_{\mathbb{R}^3}(f(u_n)-f(u))(u_n-u)\\
	&\quad-\lambda\int_{\mathbb{R}^3}\left(|u_n|^{r-2}u_n-|u|^{r-2}u\right)(u_n-u)\\
	&\geq\|u_n-u\|_E^2\\
	&\quad+b\int_{\mathbb{R}^3}\left(|\nabla u_n|^2-|\nabla u|^2\right)\int_{\mathbb{R}^3}\nabla u\cdot\nabla(u_n-u)\\
	&\quad-\int_{\mathbb{R}^3}(f(u_n)-f(u))(u_n-u)\\
	&\quad-\lambda\int_{\mathbb{R}^3}\left(|u_n|^{r-2}u_n-|u|^{r-2}u\right)(u_n-u).
	\end{split}
	\end{equation*}
	By $(f_1)$-$(f_2)$ and \eqref{eq4.5}, it is standard to show that $$\int_{\mathbb{R}^3}(f(u_n-f(u))(u_n-u)+\lambda\int_{\mathbb{R}^3}\left(|u_n|^{r-2}u_n-|u|^{r-2}u\right)(u_n-u)=o(1),$$ as $n \to \infty$.
	Moreover, by the boundedness of $\{u_n\}$ in $E$, \eqref{eq4.5} and the fact that $u_n\rightharpoonup u$ in $E$, one has $$b\int_{\mathbb{R}^3}\left(|\nabla u_n|^2-|\nabla u|^2\right)\int_{\mathbb{R}^3}\nabla u\cdot\nabla(u_n-u)=o(1),\quad\textrm{as\ }n\to\infty.$$
	Thus, from $I'_\lambda(u_n)\to0$, we conclude that $\|u_n-u\|_E^2\to0$ as $n\to\infty$ and hence the proof is completed.
\end{proof}

Note that for any $\lambda\in(0,1]$,
$$I_\lambda(u)=I(u)-\frac{\lambda}{r}\int_{\mathbb{R}^3}|u|^r\leq I(u).$$
Thus, one can follow the same line of the proof of Lemma \ref{lem3} to obtain the following result.
\begin{lemma}\label{lem4.2}
	Suppose $(f_1)$-$(f_3)$ hold and $m\geq1$. Then there exists $R_m>0$ independent of $\lambda$, such that
	$$\sup_{\mathcal{B}^{c}_{R_m}\cap E_m}I_\lambda<0.$$
\end{lemma}

\begin{lemma}\label{lem4.3}
	Assume $(f_1)$ and $(f_2)$ hold. Then for any $\lambda\in(0,1]$, there exist $\rho_\lambda>0$ and $\alpha_\lambda>0$ such that
	$$\inf_{\partial \mathcal{B}_{\rho_\lambda}}I_\lambda\geq \alpha_\lambda.$$
\end{lemma}

\begin{proof}
	By \eqref{eq3.0}, for $u\in E$, we have
	\begin{equation}\label{eq4.00}
	\begin{split}
	I_\lambda(u)&=\frac{1}{2}\|u\|^2_{E}+\frac{b}{4}\left(\int_{\mathbb{R}^3}|\nabla u|^2\right)^2-\int_{\mathbb{R}^3}F(u)-\frac{\lambda}{r}\int_{\mathbb{R}^3}|u|^r\\
	&\geq\frac{1}{2}\|u\|^2_{E}-\frac{\delta}{2}\int_{\mathbb{R}^3}|u|^2-\frac{C_\delta}{p}\int_{\mathbb{R}^3}|u|^p-\frac{\lambda}{r}\int_{\mathbb{R}^3}|u|^r.
	\end{split}
	\end{equation}
	Notice that for any $s\in[2,6]$, there exists $C(s)>0$ such that $\|u\|_s\leq C(s)\|u\|_E$. Choose $\delta$ satisfying that $C(2)\delta<\frac{1}{2}$. Then, from \eqref{eq4.00} we have
	\begin{equation*}
	I_\lambda(u)\geq\frac{1}{4}\|u\|^2_{E}-\frac{C_\delta C(p)}{p}\|u\|_E^p-\frac{\lambda C(r)}{r}\|u\|_E^r.
	\end{equation*}
	Noting that $r>p>2$, one obtains that there exist $\rho_\lambda,\alpha_\lambda>0$ such that
	$\inf_{\partial \mathcal{B}_{\rho_\lambda}}I_\lambda\geq \alpha_\lambda,$ as required.
\end{proof}

Let $\lambda\in(0,1]$, for any $u\in E$, we consider the following equation
$$-\left(a+b\int_{\mathbb{R}^3}|\nabla u|^2\right)\Delta v+V(x)v=f(u)+\lambda|u|^{r-2}u, \quad v\in E.$$
Similar to Lemma \ref{lem3.00}, one can prove that the above equation has a unique solution, denoted by $v=A_{\lambda}(u)\in E$ and the operator $A_\lambda: E\to E$ is continuous. As in Section \ref{sec3}, we shall study some properties of the operator $A_\lambda$.
\begin{lemma}\label{lem4.4}
	\
	
	\begin{itemize}
		\item [$\rm{(i)}$] $\langle I_\lambda'(u),u-A_\lambda(u)\rangle\geq\|u-A_\lambda(u)\|_E^2$ for all $u\in E$.
		\item [$\rm{(ii)}$] There exists $C>0$ independent of $\lambda$ such that $\|I_\lambda'(u)\|\leq\|u-A_\lambda(u)\|_E(1+C\|u\|_E^2)$ for all $u\in E$.
		\item [$\rm{(iii)}$] For $M>0$ and $\alpha>0$, there exists $\beta_\lambda>0$ such that $\|u-A_\lambda(u)\|_E\geq \beta_\lambda$ for any $u\in E$ with $|I_\lambda(u)|\leq M$ and $\|I_\lambda'(u)\|\geq \alpha$.
		\item [$\rm{(iv)}$] If $f$ is odd, then so is $A_\lambda$.
	\end{itemize}
\end{lemma}

\begin{proof}
	We only prove $\rm{(iii)}$, because the proofs of $\rm{(i)}$-$\rm{(ii)}$ and $\rm{(iv)}$ are similar to that of Lemma \ref{lem3.1}. Fix a number $\gamma\in(4,r)$. Notice that
	\begin{equation*}
	\langle I_\lambda'(u),u\rangle=(u-A_\lambda(u),u)_E+b\int_{\mathbb{R}^3}|\nabla u|^2\int_{\mathbb{R}^3}\nabla (u-A_\lambda(u))\cdot\nabla u,
	\end{equation*}
	which implies that
	\begin{equation*}
	\begin{split}
	I_\lambda(u)-\frac{1}{\gamma}(u-A_\lambda(u),u)_E
	&=\left(\frac{1}{2}-\frac{1}{\gamma}\right)\|u\|_E^2+\left(\frac{1}{4}-\frac{1}{\gamma}\right)b\left(\int_{\mathbb{R}^3}|\nabla u|^2\right)^2\\
	&\quad+\frac{b}{\gamma}\int_{\mathbb{R}^3}|\nabla u|^2\int_{\mathbb{R}^3}\nabla (u-A_\lambda(u))\cdot\nabla u\\
	&\quad+\int_{\mathbb{R}^3}\left(\frac{1}{\gamma}f(u)u-F(u)\right)+\left(\frac{1}{\gamma}-\frac{1}{r}\right)\lambda\|u\|_r^r.
	\end{split}
	\end{equation*}
	From $(f_1)$ and $(f_2)$, we have
	\begin{equation*}
	\begin{split}
	&\|u\|_E^2+\left(\int_{\mathbb{R}^3}|\nabla u|^2\right)^2+\lambda\|u\|_r^r\\&\quad\leq C\Bigg(|I_\lambda(u)|+\|u\|_E\|u-A_\lambda(u)\|_E
	+\int_{\mathbb{R}^3}|\nabla u|^2\left|\int_{\mathbb{R}^3}\nabla u\cdot\nabla (u-A_\lambda(u))\right|+\|u\|_p^p\Bigg).
	\end{split}
	\end{equation*}
	Then, by the H\"{o}lder inequality and Young inequality, one sees that
	\begin{equation}\label{eq4.02}
	\begin{split}
	\|u\|_E^2+\lambda\|u\|_r^r\leq C_1(|I(u)|+\|u\|_E\|u-A_\lambda(u)\|_E
	+\|u\|_E^2\|u-A_\lambda(u)\|_E^2+\|u\|_p^p).
	\end{split}
	\end{equation}
	Arguing indirectly, suppose that there exists $\{u_n\}\subset E$ with $|I_\lambda(u_n)|\leq M$ and $\|I_\lambda'(u_n)\|\geq \alpha$ such that $\|u_n-A_\lambda(u_n)\|_E\to0$ as $n\to\infty$. Then, it follows from \eqref{eq4.02} that, for large $n$,
	$$\|u\|_E^2+\lambda\|u\|_r^r\leq C_2(1+\|u\|_p^p).$$
	Similar to the proof of Lemma \ref{lem4.1}, one can show that $\{u_n\}$ is bounded in $E$. Hence, jointly with (ii), this implies $\|I'_\lambda(u_n)\|\to0$ as $n\to\infty$, a contradiction. This ends the proof.
\end{proof}

\begin{lemma}\label{lem4.5}
	There exists $\varepsilon_0>0$ independent of $\lambda$ such that for any $\varepsilon\in(0,\varepsilon_0]$, there holds
	$$A_\lambda(\overline{P}^\pm_\varepsilon)\subset P^\pm_\varepsilon.$$ Moreover, every nontrivial solutions $u\in P^+_\varepsilon$ and $u\in P^-_\varepsilon$ of \eqref{eq1} are positive and negative, respectively.
\end{lemma}

\begin{proof}
	As in the proof of Lemma \ref{lem3.2}, we have
	\begin{equation*}
	\begin{split}
	\textrm{dist}(v,P^+)\|v^-\|_E
	&\leq\int_{\mathbb{R}^3}f(u)v^-+\lambda\int_{\mathbb{R}^3}|u|^{r-2}uv^--b\int_{\mathbb{R}^3}|\nabla u|^2\int_{\mathbb{R}^3}\nabla v\cdot\nabla v^-\\
	&\leq\int_{\mathbb{R}^3}(\delta|u^-|+C_\delta|u^-|^{p-1}+|u^-|^{r-1})|v^-|\\
	&\leq C(\delta\textrm{dist}(u,P^+)+C_\delta\textrm{dist}(u,P^+)^{p-1}+\textrm{dist}(u,P^+)^{r-1})\|v^-\|_E,
	\end{split}
	\end{equation*}
	here $C$ is a constant independent of $\lambda$. Thus, choosing $\delta=\frac{1}{2C}$ and $\varepsilon_0>0$ satisfying with $CC_{\delta}\varepsilon_0^{p-2}+C\varepsilon_0^{r-2}\leq\frac{1}{4}$, for any $\varepsilon\in (0,\varepsilon_0]$, one concludes that
	$$\textrm{dist}(v,P^+)\leq\frac{1}{2}\textrm{dist}(u,P^+)<\varepsilon, \quad \textrm{for\ any\ } u\in \overline{P}_\varepsilon^+.$$
	This implies that $A_\lambda(\overline{P}^+_\varepsilon)\subset P^+_\varepsilon$. Moreover, we can show that any nontrivial solutions $u\in P^+_\varepsilon$ is positive. The other case can be proved similarly.
\end{proof}

Denote the set of critical points of $I_\lambda$ by $K_\lambda$. As in Lemma \ref{lem3.3}, we have the following results.
\begin{lemma}\label{lem4.6}
	There exists a locally Lipschitz continuous operator $B_\lambda: E\setminus K_\lambda\to E$ with the following properties:
	\begin{itemize}
		\item [$\rm{(1)}$] $B_\lambda(\overline{P}^\pm_\varepsilon)\subset P^\pm_\varepsilon$ for all $\varepsilon\in(0,\varepsilon_0]$;
		\item [$\rm{(2)}$] $\frac{1}{2}\|u-B_\lambda(u)\|_E\leq\|u-A_\lambda(u)\|_E\leq2\|u-B_\lambda(u)\|_E$ for all $u\in E\setminus K_\lambda$;
		\item [$\rm{(3)}$] $\langle I_\lambda'(u),u-B_\lambda(u)\rangle\geq\frac{1}{2}\|u-A_\lambda(u)\|_E^2$ for all $u\in E\setminus K_\lambda$;
		\item [$\rm{(4)}$] If $f$ is odd, then so is $B_\lambda$.
	\end{itemize}
\end{lemma}

Now we are ready to prove Theorem \ref{thm1.2}.

\begin{proof}[Proof of Theorem \ref{thm1.2}]
	Set
	$$G_m:=\{h\in C(\mathcal{B}_{R_m}\cap E_m,E): h \textrm{\ is\ odd\ and\ } h=\text{id} \textrm{\ on\ }\partial\mathcal{B}_{R_m}\cap E_m\},$$
	where $R_m>0$ is given by Lemma \ref{lem4.2}. For $k\geq2$, we denote $\Gamma_k$ by
	$$\Gamma_k:=\{h(\mathcal{B}_{R_m}\cap E_m\setminus Y):h\in G_m, m\geq k, Y=-Y \textrm{\ is\ open\ and\ } \gamma(Y)\leq m-k\}.$$ With the help of Lemmas \ref{lem4.1}--\ref{lem4.6}, as in Section \ref{sec3}, for any fixed $\lambda\in(0,1]$, we can define
	a minimax value $c_{k}^\lambda$ for the functional $I_\lambda(u)$ as
	$$c^\lambda_{k}:=\inf_{A\in \Gamma_k}\sup_{A\cap Q}I_\lambda, \quad k=2,3,\cdots.$$
	Moreover, one can show that for all $k\geq2$ there exists $u_k^\lambda\in Q\cap K_{c_{k}^\lambda}$ and $$0<\alpha_\lambda\leq c^\lambda_{k}\to\infty,\quad \textrm{as\ }k\to\infty,$$
	where $\alpha_\lambda$ is defined in Lemma \ref{lem4.3}. Then, we have the following
	
	{\bf Claim:} {\it For any fixed $k\geq2$, the sequence $\{u_k^\lambda\}_{\lambda\in(0,1]}$ obtained above is bounded in $E$.}
	
	Indeed, notice that for any $\lambda\in(0,1]$, $I_\lambda(u)\leq I(u)$ for $u\in E$. Then, for any given $k\geq2$, by the definition of $c^\lambda_{k}$, we can obtain that
	\begin{equation}\label{eq4.1}
	c^\lambda_{k}\leq c(m):=\sup_{\mathcal{B}_{R_m}\cap E_m}I<+\infty, \quad \textrm{for\ all\ } \lambda\in(0,1],
	\end{equation}
	where $m\geq k$ is fixed and $R_m>0$ is given by Lemma \ref{lem4.2}. Moreover, we have
	\begin{equation}\label{eq4.2}
	\begin{split}
	\frac{a}{2}\int_{\mathbb{R}^3}|\nabla u_k^\lambda|^2&+\frac{b}{4}\left(\int_{\mathbb{R}^3}|\nabla u_k^\lambda|^2\right)^2+\frac{1}{2}\int_{\mathbb{R}^3}V(x)|u_k^\lambda|^2\\&-\int_{\mathbb{R}^3}F(u_k^\lambda)-\frac{\lambda}{r}\int_{\mathbb{R}^3}|u_k^\lambda|^r=c^\lambda_{k}
	\end{split}
	\end{equation}
	\begin{equation}\label{eq4.3}
	\begin{split}
	a\int_{\mathbb{R}^3}|\nabla u_k^\lambda|^2&+b\left(\int_{\mathbb{R}^3}|\nabla u_k^\lambda|^2\right)^2+\int_{\mathbb{R}^3}V(x)|u_k^\lambda|^2\\&-\int_{\mathbb{R}^3}f(u_k^\lambda)u_k^\lambda-\lambda\int_{\mathbb{R}^3}|u_k^\lambda|^r=0
	\end{split}
	\end{equation}
	and the Pohozaev type identity
	\begin{equation}\label{eq4.4}
	\begin{split}
	\frac{a}{2}\int_{\mathbb{R}^3}|\nabla u_k^\lambda|^2+\frac{3}{2}\int_{\mathbb{R}^3}V(x)|u_k^\lambda|^2&+\frac{1}{2}\int_{\mathbb{R}^3}(DV(x),x)|u_k^\lambda|^2+\frac{b}{2}\left(\int_{\mathbb{R}^3}|\nabla u_k^\lambda|^2\right)^2\\
	&\qquad\qquad-3\int_{\mathbb{R}^3}F(u_k^\lambda)-\frac{3\lambda}{r}\int_{\mathbb{R}^3}|u_k^\lambda|^r=0.
	\end{split}
	\end{equation}
	Multiplying \eqref{eq4.2} by $\mu+6$, \eqref{eq4.3} by $-1$ and \eqref{eq4.4} by $-2$ and adding them up, one concludes that
	\begin{equation*}
	\begin{split}
	\frac{\mu+2}{2}a&\int_{\mathbb{R}^3}|\nabla u_k^\lambda|^2+\frac{\mu-2}{4}b\left(\int_{\mathbb{R}^3}|\nabla u_k^\lambda|^2\right)^2\\&+\int_{\mathbb{R}^3}\left(\frac{\mu-2}{2}V(x)-(DV(x),x)\right)|u_k^\lambda|^2
	+\int_{\mathbb{R}^3}(f(u_k^\lambda)u_k^\lambda-\mu F(u_k^\lambda))\\
	&+\frac{r-\mu}{r}\lambda\int_{\mathbb{R}^3}|u_k^\lambda|^r=(\mu+6)c^\lambda_{k}.
	\end{split}
	\end{equation*}
	Thus, noting that $2<\mu\leq p<r$, by $(V_2)$, $(f_3)$ and \eqref{eq4.1}, we see that $\int_{\mathbb{R}^3}|\nabla u_k^\lambda|^2$ and $\lambda\int_{\mathbb{R}^3}|u_k^\lambda|^r$ are bounded uniformly in $\lambda\in(0,1]$. Using this fact, from $(f_3)$, \eqref{eq4.2} and \eqref{eq4.3}, we deduce that $\{u_k^\lambda\}_{\lambda\in(0,1]}$ is bounded in $E$ and hence the proof of the claim is finished.
	
	Then, up to a subsequence, we can assume that $u_k^\lambda\rightharpoonup u_k$ in $E$ as $\lambda\to0^+$ and
	$u_k^\lambda\rightarrow u_k$ in $L^s(\mathbb{R}^3)$ for any $s\in[2,6)$.
	Note that
	\begin{equation*}
	\begin{split}
	\langle I'_\lambda(u_k^\lambda)-I'(u_k),&u_k^\lambda-u_k\rangle=\|u_k^\lambda-u_k\|_E^2+b\int_{\mathbb{R}^3}|\nabla u_k^\lambda|^2\int_{\mathbb{R}^3}\nabla u_k^\lambda\cdot\nabla(u_k^\lambda-u_k)\\
	&\quad+b\int_{\mathbb{R}^3}|\nabla u_k|^2\int_{\mathbb{R}^3}\nabla u_k\cdot\nabla(u_k^\lambda-u_k)\\
	&\quad-\int_{\mathbb{R}^3}(f(u_k^\lambda)-f(u_k))(u_k^\lambda-u_k)-\lambda\int_{\mathbb{R}^3}|u_k^\lambda|^{r-2}u_k^\lambda(u_k^\lambda-u_k)\\
	&\geq\|u_k^\lambda-u_k\|_E^2+b\int_{\mathbb{R}^3}\left(|\nabla u_k^\lambda|^2-|\nabla u_k|^2\right)\int_{\mathbb{R}^3}\nabla u_k\cdot\nabla(u_k^\lambda-u_k)\\
	&\quad-\int_{\mathbb{R}^3}(f(u_k^\lambda)-f(u_k))(u_k^\lambda-u_k)-\lambda\int_{\mathbb{R}^3}|u_k^\lambda|^{r-2}u_k^\lambda(u_k^\lambda-u_k).
	\end{split}
	\end{equation*}
	Then, by $(f_1)$-$(f_2)$ and $I'_\lambda(u_k^\lambda)=0$, as in the proof of Lemma \ref{lem4.1}, we see that $\|u_k^\lambda-u_k\|_E^2\to0$ as $\lambda\to0^+$. Consequently, $I'(u_k)=0$ and $I(u_k)=c_k:=\lim_{\lambda\to0^+}c_\lambda^k$.
	
	We claim that $u_k$ is still sign-changing. Indeed,
	using the fact that\\ $\langle I_\lambda'(u_k^\lambda),u_k^{\lambda^\pm}\rangle=0$, we have
	\begin{equation*}
	\begin{split}\|u_{k}^{\lambda^{\pm}}\|_{E}^2\leq\|u_{k}^{\lambda^{\pm}}\|_{E}^2+b\int_{\mathbb{R}^3}|\nabla u_k^\lambda|^2\int_{\mathbb{R}^3}|\nabla u^{\lambda^\pm}_k|^2&=\int_{\mathbb{R}^3}f(u_k^{\lambda^\pm})u_k^{\lambda^\pm}+\lambda\int_{\mathbb{R}^3}|u_k^{\lambda^\pm}|^{r}\\
	&\leq\int_{\mathbb{R}^3}f(u_k^{\lambda^\pm})u_k^{\lambda^\pm}+\int_{\mathbb{R}^3}|u_k^{\lambda^\pm}|^{r},
	\end{split}
	\end{equation*}
	which, jointly with \eqref{eq3.0} and Sobolev embedding theorem, implies that $\|u_k^{\lambda^\pm}\|_{E}\geq \alpha_0>0,$ where $\alpha_0$ is a constant independent of $\lambda$. From this fact and $u_k^\lambda\to u_k$ in $E$, it is easy to see that $\|u_k^\pm\|_{E}\geq \alpha_0$. Therefore, $u$ is a sign-changing solution of \eqref{eq1}.
	
	Noting that $c_\lambda^k$ is nonincreasing in $\lambda$, we see that $c_k\geq c^\lambda_{k}$ for any $\lambda\in(0,1]$. Since $c^\lambda_{k}\to\infty$ as $k\to\infty$, one deduces that $$\lim_{k\to\infty}c_k=\infty.$$
	This implies that equation \eqref{eq1} admits infinitely many sign-changing solutions and therefore the proof is completed.
\end{proof}

%\section*{Acknowledgments} We would like to thank the referee for
%his/her valuable comments and helpful suggestions which have led to
%an improvement of the presentation of this paper.

\end{document}